\documentclass[12pt]{amsart}
\usepackage[all]{xy}
\usepackage{amsmath,dsfont,mathtools,enumitem, amsfonts, amssymb}
\usepackage[pdftitle={K-regularity of locally convex algebras},
  pdfauthor={Hvedri Inassaridze},
  pdfsubject={Mathematics}
]{hyperref}

\usepackage[lite]{amsrefs}
		
\usepackage{microtype}
\usepackage{wasysym}

\newtheorem{thm}{Theorem}[section]
\newtheorem{cor}[thm]{Corollary}

\newtheorem{prop}[thm]{Proposition}
\theoremstyle{definition}
\newtheorem{defn}[thm]{Definition}
\theoremstyle{remark}
\newtheorem{rem}[thm]{Remark}

\numberwithin{equation}{section}

\begin{document}

\author={Hvedri Inassaridze}

\title {algebraic K-functors for \Gamma-rings}

\address{A.~Razmadze Mathematical Institute of Tbilisi State University, 6, Tamarashvili Str., Tbilisi 0179, Georgia.}
\email{inassari@gmail.com}

\thanks{}

\section{title}
{Algebraic K-functors for $\Gamma$-rings}

\section{abstract}
This is an attempt to extend to algebraic K-theory our approach to group actions in homological algebra that could be called introduction to $\Gamma$-algebraic K-theory. For $\Gamma$-rings the Milnor algebraic K-theory and Swan's algebraic K-functors are introduced and investigated. Particularly, the Matsumoto conjecture related to symbol group, and the Milnor conjectures related to  Witt algebras and Chow groups are extended and proved.

\section{introduction}

We continue the study of our approach to group actions in homological algebra for the case of algebraic K-theory, that was started in [8] and continued in [9,10]. The originality of our approach to the study of homological properties of groups and rings consists of the definition of a new and natural, unexpected action of a group $\Gamma$ on the classical chain complexes defining classical homology of a group $G$ and of a ring $R$ respectively, induced by its given action on the group $G$ and on the ring $R$ respectively. By taking the homology groups of the tensor product of the classical chain complex under this action of  $\Gamma$ with the coefficient group, we have defined new homology and cohomology groups of $\Gamma$-groups and $\Gamma$-rings, called $\Gamma$-equivariant (co)homology. These new tools allowed us to extend important homological properties of groups and rings to the case of $\Gamma$-groups and $\Gamma$-rings, opening a new direction in homological algebra that we called $\Gamma$-homological algebra. The recently developed investigation of equivariant cohomological dimensions is closely related to $\Gamma$-homological algebra [3,4].

We are trying to investigate the influence of group actions on the foundations of algebraic K-theory [13,2,16,17] when a discrete group is acting on the basic ring.

\;

Some notations that will be used throughout the paper:

$\Gamma G$ denotes the normal subgroup of $G$ generated by the set of elements $^{\gamma}gg^{-1},g\in G$, $\gamma \in \Gamma$.
The quotient group $G/\Gamma G$ is denoted  $G_\Gamma$.

$\Gamma-[G,G]$ denotes the subgroup of the $\Gamma$-group $G$ generated by the commutant subgroup [$G$,$G$] and the elements of the form $^{\gamma}gg^{-1}$, $g\in G$, $\gamma \in \Gamma$. It is called the $\Gamma$-commutant of the group $G$.

$G^{ab}_\Gamma$ denotes the abelianization of the group $G_\Gamma$.

For an augmented simplicial group $G^{+}_{\ast}:G_{\ast}\overset{\tau}\rightarrow G_{-1}$, the notation $\pi_{-1}(G^{+}_{\ast})$ equal to $Coker \tau$ is used.

${\mathbf{G}}^\Gamma$  and ${\mathbf{R}}^\Gamma$ denote the category whose objects are groups and rings respectively on which a fixed group $\Gamma$ is acting, called $\Gamma$-groups and $\Gamma$-rings respectively, and morphisms are homomorphisms compatible with the action of $\Gamma$.

\;


\section{Preliminaries}

We give some important examples of $\Gamma$-rings with unit.

Let $A$ be an arbitrary ring with unit and $\Gamma$ be a group. It is said that $A$ is a $\Gamma$-ring, or $\Gamma$ is acting on $A$,  if there is given a group homomorphism
$\Gamma \rightarrow Aut(A)$.

Rings having the identity automorphism only: the commutative rings $\mathds{Z}$, $\mathds{Q}$, $\mathds{R}$, $\mathds{Z}/p\mathds{Z}$, $\mathds{Z}/(p)$ (the ring of integers localized at p a rational prime), and $\mathds{Q}_{p}$ (the ring of p-adic numbers).

The fields are important examples of $\Gamma$-rings. For instance, the field $\mathds{C}$ of complex numbers has two automorphisms, the identity and the complex conjugation sending a + bi  to  a - bi;
the field $F_q$ of prime power $q=p^n$ has the Frobenius automorphism $\phi: x \rightarrow x^p$, and the automorphism group is cyclic of order $n$ generated by $\phi$; the quadratic field $\mathds{Q}(2)$ with automorphism group of order 2; the Galois groups of field extensions.

The non-commutative ring $A$ with non-trivial group $A^*$ of invertible elements (called also units)(for instance the quaternion algebra). Then   $A^*$ is acting on $A$  sending $x \rightarrow axa^-1$, $x\in A$, $a\in A^*$.

Let $RA$ be an $R$-algebra, $R$ is a commutative ring with unit. The group $\Gamma$ is acting on $RA$ if it is acting on $R$ and on $A$, and one has
$^{\gamma}(ra)= ^{\gamma}(r)^{\gamma}(a)$. For instance, the symmetric group $S_n$ is acting on the polynomial $R$-algebra $R[x_1, ... ,x_n]$, acting trivially on $R$
 and naturally acting on $Z[x_1, ... ,x_n]$.

 Let R(G) be a group ring and $\Gamma$ be a group acting on the group G and on the ring R.Then R(G) becomes a $\Gamma$-ring with the action $^{\gamma}(rg) = ^{\gamma}
 r^{\gamma} g,  r\in R, g\in G, \gamma \in \Gamma.$

Now, we recall some definitions and propositions given in [8,10] which will be used later.

Any exact sequence E of $\Gamma$-groups
\begin{equation}
E: 1\rightarrow A\rightarrow B\overset{\tau}\rightarrow G\rightarrow 1
\end{equation}
is called $\Gamma$-extension of the $\Gamma$-group $G$ by the $\Gamma$-group $A$. The extension $E$ is an extension with $\Gamma$-section map if there is a map $\beta: G\rightarrow X$ such that $\tau\beta = 1_G$, and compatible with the action of $\Gamma$. In addition, if $\beta$ is a homomorphism then the extension E is called split extension.

 A $\Gamma$-equivariant $G$-module $A$ is a $G$-module equipped with a $\Gamma$-module structure and the actions of $G$ and $\Gamma$ are related to each other by the equality
$$
^{\sigma}(^{g}a)=^{\sigma}g(^{\sigma} a)
$$
for $g\in G$, $\sigma\in \Gamma$,$a\in A$.

The category of $\Gamma$-equivariant $G$-modules is equivalent to the category of $G\rtimes \Gamma$-modules, where $G\rtimes \Gamma$ is the semi-direct product of G and $\Gamma$.

If $E$ is an extension with $\Gamma$-section map and $A$ is a $\Gamma$-equivariant $G$-module it is called  $\Gamma$-equivariant extension of $G$ by $A$. In addition, if $X$ and $G$ are $\Gamma$-equivariant G-modules it is called proper sequence of $\Gamma$-equivariant $G$-modules.

 A $\Gamma$-equivariant $G$-module $F$ is called a relatively free $G\rtimes\Gamma$-module if it is a free $G$-module with basis a $\Gamma$-set, and a relatively projective $\Gamma$-equivariant $G$-module is a retract of a relatively free $\Gamma$-equivariant G-module.

The class $\mathcal{P}$ of relatively projective $\Gamma$-equivariant $G$-modules is a projective class with respect to proper sequences of $\Gamma$-equivariant $G$-modules.

For the cohomological description of the set $E^{1}_{\Gamma}(G,A)$ of equivalence classes of $\Gamma$-equivariant extensions of $G$ by $A$ the $\Gamma$-equivariant homology and cohomology of $\Gamma$-groups have been introduced as relative $Tor_n^{\mathcal{P}}$ and $Ext^n_{\mathcal{P}}, n \geq 0$, in the category of $\Gamma$-equivariant $G$-modules, namely
the $\Gamma$-equivariant homology and cohomology of $\Gamma$-groups are defined as follows
$$
 H_n^{\Gamma}(G,A)= Tor_n^{\mathcal{P}}(\mathbb{Z},A), H^n_{\Gamma}(G,A)= Ext^n_{\mathcal{P}}(\mathbb{Z},A), n\geq 0,
$$
where the functors $\bigotimes$ and $Hom$ are taken over the ring $\mathbb{Z}(G\rtimes \Gamma)$ and the groups $G$ and $\Gamma$ are trivially acting on the abelian group $\mathbb{Z}$ of integers. This cohomology is the relative group cohomology in the sense of Hochschild and Adamson [6,1].

 A $\Gamma$-group is called $\Gamma$-free if it is a free group with basis a $\Gamma$-set.

 Any free group $F$($G$) generated by a $\Gamma$-group $G$ becomes a $\Gamma$-free group by the following action of $\Gamma$: $^\gamma{\mid g \mid} =  \mid ^\gamma{g} \mid $, $g\in G, \gamma \in \Gamma$. The defining property of the $\Gamma$-free group $F$ with basis $E$ is that every $\Gamma$-map $E\overset{f}\rightarrow G$ to a $\Gamma$-group $G$ is uniquely extended to a $\Gamma$-homomorphism $F\overset{f'}\rightarrow G$.

 Let $\mathbb{F}$ be the projective class of $\Gamma$-free groups in the category ${\mathbf{G}}^\Gamma$ of $\Gamma$-groups.

 There are isomorphisms
$$
H_n^{\Gamma}(G,A) \cong L^{\mathbb{F}}_{n-1}(I(G)\otimes_{G\ltimes \Gamma}A), H^n_{\Gamma}(G,A)\cong R^{n-1}_{\mathbb{F}}Der_{\Gamma}(G,A)
$$
for $n\geq 2$, where I(G) is the kernel of the natural homomorphism $\mathbb{Z}(G)\rightarrow \mathbb{Z}$ of $\Gamma$-equivariant G-modules, $Der(G,A)$ is the group of $\Gamma$-derivations, $L^{\mathbb{F}}_{n-1}$ and R$^{n-1}_{\mathbb{F}}$ denote respectively the left and right derived functors with respect to the projective class $\mathbb{F}$.

Let
$$
1\rightarrow A\rightarrow B\overset{\tau}\rightarrow G\rightarrow 1
$$
be a short exact sequence of $\Gamma$-groups with $\Gamma$-section map and $\alpha: P\rightarrow B$ be a $\Gamma$-projective presentation of the $\Gamma$-group B. Then there is an exact sequence
$$
0\rightarrow V\rightarrow H^{\Gamma}_2(B)\rightarrow H^{\Gamma}_2(G)\rightarrow A/[B,A]_{\Gamma}\rightarrow H^{\Gamma}_1(B)\rightarrow H^{\Gamma}_1(G)\rightarrow 0,
$$
where V is the kernel of the $\Gamma$-homomorphism $[P,S]_{\Gamma}/[P,R]_{\Gamma}\rightarrow [B,A]_{\Gamma}$ induced by $\alpha$, R = Ker $\alpha$ and S = Ker $\tau\alpha$.

If $G$ is a $\Gamma$-group, then

(1)
$H^{\Gamma}_1(G,A) = G/[G,G]_{\Gamma}\otimes A$,
$G$ and $\Gamma$ are trivially acting on $A$.

(2)
$H^{\Gamma}_2(G)$ is isomorphic to the group $(R\cap [P,P]_{\Gamma})/[P,R]_{\Gamma}$, where $R = Ker \alpha$ and $\alpha: P\rightarrow G$ is a $\Gamma$-projective presentation of $G$ (Hopf formula for the $\Gamma$-equivariant homology of groups).

\

\section{$\Gamma$-algebraic K-functors}

\begin{prop}
Let $\tau: G\rightarrow G'$ be a surjective homomorphism of $\Gamma$-groups. Then one has the following exact sequence of groups:
$$
0\rightarrow Ker\Gamma(\tau)/\Gamma Ker\tau \rightarrow (Ker\tau)_{\Gamma}\rightarrow G_{\Gamma}\rightarrow G'_{\Gamma}\rightarrow 0.
$$
\end{prop}

\begin{proof}
Consider the following commutative diagram of exact columns and rows:

\begin{equation}
\xymatrix{
 & 0 \ar[d] & 0 \ar[d] & 0 \ar[d]\\
0 \ar[r] & Ker \Gamma(\tau) \ar[d]\ar[r] &  \Gamma G \ar[d]\ar[r]^{\Gamma(\tau)} & \ar[d]\ar[r]   \Gamma G' \ar[d]\ar[r] & 0\\
0 \ar[r] & Ker\tau \ar[d]\ar[r] & G \ar[d]\ar[r]^{\tau}  & G' \ar[d]\ar[r] & 0\\
0 \ar[r] & Ker\tau_{\Gamma}\ar[r] \ar[r] \ar[d] & G_{\Gamma} \ar[d]\ar[r]^{\tau_\Gamma} & G'_{\Gamma} \ar[d]\ar[r] & 0\\
 & 0 & 0 & 0 &
}
\end{equation}

One gets $\Gamma Ker\tau  \subseteq Ker\Gamma(\tau)$, the isomorphism $Ker\tau/Ker\Gamma(\tau) \cong Ker\tau_{\Gamma}$ and the surjection $Ker\tau/\Gamma Ker\tau \rightarrow Ker\tau/Ker\Gamma(\tau)$, implying the exact sequence  $e\rightarrow Ker\Gamma(\tau)/\Gamma Ker\tau \rightarrow (Ker\tau)_{\Gamma} \rightarrow Ker\tau_{\Gamma} \rightarrow e.$

According to the diagam 5.1, we finally obtain the required exact sequence.
\end{proof}
\begin{cor}
If $G \overset{\alpha}\rightarrow G'\overset{\beta}\rightarrow G''\rightarrow e$ is an exact sequence of $\Gamma$-groups, then the sequence $G_{\Gamma} \overset{\alpha_{\Gamma}}\rightarrow G'_{\Gamma}\overset{\beta_{\Gamma}}\rightarrow G''_{\Gamma}\rightarrow e$ is exact.
\end{cor}
Taking into account the surjection $G_{\Gamma}\rightarrow (Ker\beta)_{\Gamma}$, follows immediately from Proposition 5.1.

Later we will need the following formula related to exact sequences of $\Gamma$-groups. Let $G$ be a $\Gamma$-group an $H$ be its normal $\Gamma$-subgroup. Consider the exact sequence of $\Gamma$-groups: $e\rightarrow H \overset{\sigma}\rightarrow G\rightarrow G / H\rightarrow e$. It induces the short exact sequence: $(H)_{\Gamma} \overset{\sigma_{\Gamma}}\rightarrow (G)_{\Gamma}\rightarrow (G / H)_{\Gamma}\rightarrow e$ implying the isomorphism $Coker \sigma_{\Gamma} \cong (G / H)_{\Gamma}$. On the other hand, one has the equality
$Coker \sigma_{\Gamma} = G / H\cdot \Gamma G$, and finally one gets the needed formula
$$
(G / H)_{\Gamma} = G / (H\cdot \Gamma G).
$$
 We are now going to define algebraic K-functors of a ring $A$ with unit on which a group $\Gamma$ is acting and called $\Gamma$-ring.

It is easily checked that this action induces the action on the n-th general group $GL_{n}(A)$ of $n \times n$ invertible matrices over A given by $^{\gamma}\parallel a_{ij} \parallel =  \parallel ^{\gamma}a_{ij} \parallel$, and the n-th elementary group $E_{n}(A)$ of elementary $n \times n$ - matrices $e_{ij}(a), i,j\leq n, a\in A$, over $A$ becomes a $\Gamma$-subgroup of  $GL_{n}(A)$. Therefore one gets a natural action of the group $\Gamma$ on the group $K_{1}(n,A) =  GL_{n}(A) / E_{n}(A).$.

\begin{defn}
For the unital ring A the non-stable $\Gamma$ - algebraic K-functor $K^{\Gamma}_{1}(n,A)$ is defined as $(K_{1}(n,A))_{\Gamma}$.
\end{defn}

Consider the exact sequence $e\rightarrow E_{n}(A)\overset{\nu_{n}}\rightarrow GL_{n}(A)\rightarrow K_{1}(n,A)\rightarrow 0$.  According to Corollary 4.2, it induces the exact sequence

$e\rightarrow (E_{n}(A))_{\Gamma}\rightarrow (GL_{n}(A))_{\Gamma}\rightarrow (K_{1}(n,A))_{\Gamma}\rightarrow 0$.

Passing to the direct limit,

$GL(A) = (GL_{n}(A))_{n\rightarrow \infty},  E(A) = (E_{n}(A))_{n\rightarrow \infty}, K_{1}(A) = (K_{1}(n,A))_{n\rightarrow \infty}$,
one gets the action of $\Gamma$ on $GL(A)$, $E(A)$ and $K_{1}(A)$ respectively,

and  the isomorphisms $(GL(A))_{\Gamma}\approx ((GL_{n}(A))_{\Gamma})_{n\rightarrow \infty}, (E(A))_{\Gamma}\approx ((E_{n}(A))_{\Gamma})_{n\rightarrow \infty}, (K_{1}(A))_{\Gamma}\approx ((K_{1}(n,A))_{\Gamma})_{n\rightarrow \infty}$.

Finally, we obtain the following exact sequence

$(E(A))_{\Gamma}\overset{\nu_{\Gamma}}\rightarrow (GL(A))_{\Gamma}\rightarrow K^{\Gamma}_{1}(A)\rightarrow 0$.

\begin{defn}
For the unital ring A the $\Gamma$ - algebraic K-functor $K^{\Gamma}_{1}(A)$ is defined as $(K_{1}(A))_{\Gamma}$.
\end{defn}

One has isomorphisms $K^{\Gamma}_{1}(A)\cong GL(A)/(E(A)\cdot \Gamma GL(A))\cong H^{\Gamma}_{1}(GL(A))$.

Let A be a commutative $\Gamma$-ring with unit, and $A^{*}$ its multiplicative subgroup of invertible elements. There is the naturally splitting exact sequence
$$
 e\rightarrow SL(A)\rightarrow GL(A)\overset{det}\rightarrow A^{*}\rightarrow 0,
$$

where $det$ denotes the determinant map which is a $\Gamma$-homomorphism, SL(A) is the kernel of det, and the splitting map $A^{*}\rightarrow GL(A)$ is also a $\Gamma$-homomorphism.

This exact sequence induces the splitting exact sequence of $\Gamma$-groups,
$$
0\rightarrow SL(A) / E(A)\rightarrow K_{1}(A)\rightarrow A^{*}\rightarrow 0
$$

and the isomorphism of $\Gamma$-groups $K_{1}(A)\overset{\approx}\rightarrow SL(A) / E(A) \oplus A^{*}$, where $\Gamma$ is acting componentwise on the right side.

Finally, we obtain the following

\begin{prop}

If A is a commutative unital $\Gamma$-ring, there is an isomorphism

$$
K^{\Gamma}_{1}(A)\approx (SL(A) / E(A))_{\Gamma} \oplus (A^{*})_{\Gamma}\cong (SL(A) / (E(A)\cdot \Gamma SL(A))) \oplus (A^{*})_{\Gamma},
$$

and $K^{\Gamma}_{1}(F) = F^{*}_{\Gamma}$ for a field $F$.

\end{prop}

For instance,  if we consider the nontrivial action of the cyclic group $\mathds{Z}_{2}$ of order 2 on the field $\mathds{C}$ of complex numbers, then $K^{\mathds{Z}_{2}}_{1}(\mathds{C})= \mathds{C}^{\ast} / \mathds{Z}_{2} \mathds{C}^{\ast}$, and the subgroup $\mathds{Z}_{2} \mathds{C}^{\ast}$ is generated by the elements $(a-bi)(a+bi)^{-1}, a+bi\neq 0$.

Let $St_{n}(A)$ be the $n$-th Steinberg group of the ring $A$ with generators $x^{n}_{ij}(a)$ and $i,j\leq n$, $a,b\in A$, satisfying the relations

$x^{n}_{ij}(a) . x^{n}_{ij}(b) = x^{n}_{ij}(a+b), i,j\leq n,$

$[x^{n}_{ij}(a) . x^{n}_{kl}(b)] = 1, i\neq l, k\neq j,$

$[x^{n}_{ij}(a) . x^{n}_{jk}(b)] = x^{n}_{ik}(ab), i\neq k.$

The given action of $\Gamma$ on $A$ induces the following action on the n-th Steiberg group $St_{n}(A)$ as $^{\gamma}(x^{n}_{ij}(a)) = x^{n}_{ij}(^{\gamma}a)$ which is compatible with its defining relations.

Passing to the direct limit,  $St(A) = (St_{n}(A))_{n\rightarrow \infty}$, one gets the action of $\Gamma$ on $St(A)$, and the isomorphism $(St(A))_{\Gamma}\approx ((St_{n}(A))_{\Gamma})_{n\rightarrow \infty}$.

It is evident that the homomorphism $St_{n}(A)\overset{\phi_{n}}\rightarrow E_{n}(A)$, sending $x^{n}_{ij}(a)$ $to$ $e^{n}_{ij}(a)$, is a homomorphism of $\Gamma$-groups. and $K_{2}(n,A)$ is a $\Gamma$-subgroup of $St_{n}(A)$.

Therefore, the well-known exact sequence $0\rightarrow K_{2}(n,A)\rightarrow St_{n}(A)\rightarrow E_{n}(A)\rightarrow 0$ induces the exact sequence

$(K_{2}(n,A))_{\Gamma}\rightarrow (St_{n}(A))_{\Gamma}\rightarrow (E_{n}(A))_{\Gamma}\rightarrow 0$.

After the direct limit, we finally obtain the following exact sequence

$(K_{2}(A))_{\Gamma}\rightarrow (St(A))_{\Gamma}\overset{\phi_{\Gamma}}\rightarrow (E(A))_{\Gamma}\rightarrow 0$.

\begin{defn}
For the unital ring A the $\Gamma$ - algebraic K-functor $K^{\Gamma}_{2}(A)$ is defined as $Ker \phi_{\Gamma}$.

\end{defn}

It is easily checked one has the natural surjection $(K_{2}(A))_{\Gamma}\rightarrow K^{\Gamma} _{2}(A)$, and the equality $\Gamma(K_{2}(A)) = K_{2}(A) / K_{2}(A)\cap \Gamma (St(A))$ holds.

 It follows immediately that $K^{\Gamma} _{2}(F) = 0$ for finite field F.

Let A be an unital ring and $Sym(A)$ be the symbol group of the ring $A$ generated by the elements $\{u,v\}$, $u,v \in A^{*}$, satisfying the following relations

$(S_{0})$     $\langle u,1-u\rangle = 1, u\neq 1, 1-u \in A^{*},$

$(S_{1})$     $\langle uu',v\rangle = \langle u,v\rangle \langle u',v\rangle$,

$(S_{2})$     $\langle u,vv'\rangle = \langle u,v\rangle \langle u,v'\rangle$.

where $A^{*}$ denotes the multiplicative group of invertible elements of the ring A [5]. By Matsumoto's theorem, the groups $Sym(A)$ and $K_{2}(A)$ are isomorphic when A is a field [11].

If $A$ is a $\Gamma$-ring, the action of $\Gamma$ induces a natural action on the group $Sym(A)$ given by $^{\gamma}\langle u,v\rangle = \langle^{\gamma}u,^{\gamma}v\rangle, u,v\in A^{\star}$.
The $\Gamma$-equivariant symbol group $Sym^{\Gamma}(A)$ is defined as $Sym^{\Gamma}(A) = (Sym(A))_{\Gamma}$.

In particular, if $\Gamma$ is the group $A^{*}$ of invertible elements of $A$, the action of $A^{*}$ on $A$ by conjugation induces an action on $Sym(A)$ given by $^{\gamma}\langle u,v\rangle = \langle \gamma u \gamma ^{-1},\gamma v \gamma ^{-1}\rangle , \gamma, u,v\in A^{\star}$. One gets the $A^{*}$-equivariant symbol group $Sym^{A^{*}}(A) = (Sym(A))_{A^{*}}$.

Now, Matsumoto's theorem for $\Gamma$-fields will be given. Namely,

\begin{thm}
If F is a $\Gamma$-field, then the sequence
\begin{equation}
0\rightarrow Ker \Gamma(\phi) / \Gamma Ker(\phi)\rightarrow Sym^{\Gamma}(F)\rightarrow K^{\Gamma}_{2}(F)\rightarrow 0,
\end{equation}
is exact, where $\phi$ denotes the canonical surjection $St(F)\rightarrow E(F)$.
\end{thm}

$Proof$. As we see, the action of the group $\Gamma$ on the field $F$ induces its action on $Sym(A)$ and also on $K_{2}(F)$ as a subgroup of the Steinberg group $St(F)$. It is easily checked that the Matsumoto's isomorphism $Sym(F) \cong K_{2}(F)$ is compatible with the action of $\Gamma$. Therefore, this implies the isomorphism $(Sym(F))_{\Gamma} \cong (K_{2}(F))_{\Gamma}$. It remains to apply Proposition 5.1 for the homomorphism $\phi$. This completes the proof.

Let $A$ be a non-commutative local ring such that $A/Rad(A) \neq F_{2}$. Consider A as a $\Gamma$-ring, where $\Gamma$ is the group $A^{\ast}$ of units acting on A by conjugation.

 There is an exact sequence [5,10]
 $$
 0\rightarrow K_{2}(A)\rightarrow U(A)\overset{\tau}\rightarrow [A^{*}, A^{*}]\rightarrow e,
 $$
 where the group $U(A)$ is generated by the elements $\langle u,v\rangle$,$u,v \in A^{*}$, satisfying the following relations

$(U_{0})$        $\langle u, 1-u\rangle = 1, u\neq 1, 1-u \in A^{*},$

$(U_{1})$        $\langle uv,w \rangle = ^{u}\langle v,w \rangle \langle u,w\rangle$,

$(U_{2})$        $\langle u,vw\rangle \langle v,wu\rangle \langle w,uv\rangle = 1$,

and $^{u}\langle v,w \rangle = \langle uvu^{-1},uwu^{-1} \rangle$ [5].

Generalization of Matsumoto's theorem for non-commutative local $\Gamma$-rings.

\begin{thm}
If A is a  non-commutative local $A^{*}$-ring such that $A/Rad(A) \neq F_{2}$, there is an exact sequence
$$
0\rightarrow Ker \Gamma(\tau) / \Gamma(K_{2}(A))\rightarrow (K_{2}(A))_{A^{*}}\rightarrow Sym^{A^{*}}(A)\rightarrow
$$
$$
[A^{*}, A^{*}] / \Gamma([A^{*}, A^{*}])\rightarrow 0,
$$
where $A^{*}$ is the group of units of $A$ acting on $A$ by conjugation.
\end{thm}

$Proof$. It is easily checked that the group $U(A)$ becomes $A^{*}$-group with respect to the action $^{\gamma}\langle v,w \rangle = \langle \gamma v \gamma^{-1},\gamma w \gamma^{-1} \rangle$    and by results of [5,10] there is a surjective $A^{*}$-homomorphism $U(A)\rightarrow Sym(A)$. Moreover, it is proven that the group $(Sym(A))_{A^{}\star}$ is abelian and isomorphic to $(U(A))_{A^{\star}}$. One gets the exact sequence $(K_{2}(A))_{\Gamma}\rightarrow (U(A))_{\Gamma}\overset{\tau_{\Gamma}}\rightarrow [A^{*}, A^{*}] / \Gamma([A^{*}, A^{*}])$. By applying Proposion 5.1 again for the homomorphism $\tau$, we obtain the needed exact sequence.

\begin{defn}
Let $\mathds{A}$ be an arbitrary category. It is said a group $\Gamma$ is acting on $\mathds{A}$, if for any pair (A,B) of objects of $\mathds{A}$
the corresponding set H(A,B) of morphisms is a $\Gamma$-set such that for any $f:A\rightarrow B$ and $g:B\rightarrow C$ the following equality hold:
 $^{\gamma}(gf) = ^{\gamma}(g)^{\gamma}(f)$, and $^{\gamma}(1_{A}) = 1_{A}$ for any unit morphism $1_{A}: A\rightarrow A$ of the category $\mathds{A}$,
 and any element $\gamma\in \Gamma$. This category will be denoted $\mathds{A}^{\Gamma}$ and called $\Gamma$-category.
\end{defn}

Examples.

The category $\mathcal{G}r^{\Gamma}$ of $\Gamma$-groups is the category of groups on which the group $\Gamma$ is acting. The category $\mathcal{R}^{\Gamma}$ of $\Gamma$-rings
is the category of rings on which the group $\Gamma$ is acting.

\begin{defn}
It is said that a covariant functor $T:\mathds{A}^{\Gamma}\rightarrow \mathcal{G}r^{\Gamma}$ is given, if T is a covariant functor from the category $\mathds{A}$
to the category $\mathcal{G}r$ of groups such that $T(^{\gamma}(f)) = ^{\gamma}(T(f))$ for any morphism f of $\mathds{A}$ and any element $\gamma$ of $\Gamma$.
\end{defn}

\begin{defn}
It is said that $(\mathrm{F}, \tau, \delta)$ is a cotriple of the $\Gamma$-category $\mathds{A}^{\Gamma}$ if it is a cotriple of the category $\mathds{A}$, the functor $F$ is compatible with the action of $\Gamma$, and the morphisms $\tau, \delta$ are $\Gamma$-morphisms.
\end{defn}

Let $(\mathrm{F}, \tau, \delta)$ be a cotriple on the category  $\mathds{A}^{\Gamma}$, where $\tau:F(A)\rightarrow A, \delta:F(A)\rightarrow F(F(A))$. It induces the augmented simplicial object

$$
F^{+}_{\star}(A) = F_{\star}(A)\overset{\tau}\rightarrow A,
F_{0}(A) = F(A), F_{n} = F(F_{n-1}(A)) for \geq 1,
$$

$
\lambda^{n}_{i} = F_{i}\tau F_{n-i},
s^{n}_{i} = F_{i}\delta F_{n-i}
$
for $0\leq i\leq n$.

\;

Let $L_{\ast} = (L_{n}, \lambda^{n}_{i}, s^{n}_{i}, n\geq 0, n\in \mathds{Z})$ be a simplicial group with bord and degeneracy operators  $\lambda^{n}_{i}$ and $s^{n}_{i}$ respectively.

\begin{defn}
It is said that the group $\Gamma$ is acting on the simplicial group $L_{\ast}$, if it is acting on the groups $L_{n}, n\geq 0, n\in \mathds{Z}$,
 such that the homomorphisms $\lambda^{n}_{i}, s^{n}_{i}, n\geq 0, n\in \mathds{Z}$, are $\Gamma$-homomorphisms. The quotient simplicial group $L^{\Gamma}_{\ast}$ of $L_{\ast}$,
$L^{\Gamma}_{\ast} = ((L_{n})_{\Gamma}, \lambda^{n}_{i,\Gamma}, s^{n}_{i,\Gamma}, n\geq 0, n\in \mathds{Z})$, where the homomorphisms $\lambda^{n}_{i,\Gamma}$ and $s^{n}_{i,\Gamma}$ are induced by $\lambda^{n}_{i}$ and $s^{n}_{i}$ respectively, is called Connes simplicial group of $L_{\ast}$ with respect to the action of $\Gamma$.

 The action of $\Gamma$ on the simplicial group $L_{\ast}$ induces its action on the homotopy groups of $L_{\ast}$, and the homotopy groups $\pi_{n}(L^{\Gamma}_{\ast}), n\geq 0, n\in \mathds{Z}$, of the Connes simplicial group $L^{\Gamma}_{\ast}$ of $L_{\ast}$, are called $\Gamma$-equivariant homotopy groups of the simplicial group $L_{\ast}$.
\end{defn}

It is evident that the action of the group $\Gamma$ on pseudo-simplicial groups [7] is defined analogously.

A similar definition of the group action on a chain complex of modules was given in [10]. Its quotient chain complex could be called Connes chain complex. The homology groups of the quotient chain complex were called $\Gamma$-equivariant homology groups of the chain complex.

\begin{rem}

Connes complex firstly appear for Hochschild complex under the cyclic group action [11] which is the action of the group of integers on the Hochschild complex [10].
\end{rem}
One has the exact sequence

$$
... \rightarrow \pi_{n+1}(L^{\Gamma}_{\ast})\rightarrow \pi_{n}(\Gamma L_{\ast})\rightarrow \pi_{n}(L_{\ast})\rightarrow \pi_{n}(L^{\Gamma}_{\ast})\rightarrow \pi_{n-1}(\Gamma L_{\ast})\rightarrow ...,
$$

inducing the exact sequence

\begin{equation}
\pi_{n}(\Gamma L_{\ast})\rightarrow (\pi_{n}(L_{\ast}))_{\Gamma}\rightarrow \pi_{n}(L^{\Gamma}_{\ast})\rightarrow \pi_{n-1}(\Gamma L_{\ast}).
\end{equation}

The action of the group $\Gamma$ on the simplicial group $L_{\ast}$ induces an action of $\Gamma$ on the first sequence implying the second exact sequence. It is evident the sequence similar to the sequenc 5.3 holds for the augmented $\Gamma$-simplicial group $L^{+}_{\ast}$.

\begin{defn}
Let $T$ be a covariant functor from the $\Gamma$-category $\mathds{A}^{\Gamma}$ to the category $\mathcal{G}r^{\Gamma}$ of $\Gamma$-groups. The non-abelian left $\Gamma$-derived functors $L^{\Gamma}T_{n}(A)$ of $T$ with respect to the cotriple $\mathds{F} = (\mathrm{F}, \tau, \delta)$ are defined as follows:
$$
 L^{\Gamma}T_{n}(A) = \pi_{n}((TF_{\ast}(A))^{\Gamma}), n\geq 0.
$$
\end{defn}

Non-abelian $\Gamma$-derived functors of the functor $T$ could also be defined with respect to the projective class of the category $\mathds{A}^{\Gamma}$ in the sense of Eilenberg - Moore by using $\Gamma$-pseudo-simplicial groups. All these definitions of non-abelian $\Gamma$-derived functors generalized the well-known non-abelian derived functors of covariant functors when the group $\Gamma$ is acting trivially [7].

Let $G$ be a $\Gamma$-group and $A$ be a $\Gamma$-ring, and let $F(G)$ and $F(A)$ be the free group and the free ring, generated by the set $(|g|), g\in G$, and by the set $(|a|), a\in A$, respectively. Define the action of the group $\Gamma$ on $F(G)$ by $^{\gamma}|g| = |^{\gamma}g|$, and on $F(A)$ by $^{\gamma}|a| = |^{\gamma}a|, \gamma\in \Gamma$. Then the free cotriples induced by $G$ and $A$ become cotriples of the categories $\mathcal{G}r^{\Gamma}$ and $\mathcal{R}^{\Gamma}$ respectively.

Let $A$ be a unital $\Gamma$-ring and $GL(A)$ the induced $\Gamma$-group. We are going to define Swan's $\Gamma$-algebraic K-functors as $\Gamma$-derived functors of the functor $Gl(-)$ with respect to the free cotriple, namely

\begin{defn}
For a $\Gamma$-ring $A$ the Swan's $\Gamma$-algebraic K-functors $K_{n}^{S,\Gamma}(A)$ are defined as follows

$K_{n}^{S,\Gamma}(A) =  \pi_{n-2}(GLF_{\ast}(A))^{\Gamma}$ for $n\succeq 3$,

and the groups $K_{1}^{S,\Gamma}(A)$, $K_{2}^{S,\Gamma}(A)$ are defined by the exact sequence
$0\rightarrow K_{2}^{S,\Gamma}(A)\rightarrow  \pi_{0}(GLF_{\ast}(A))^{\Gamma}\rightarrow (GL(A))_{\Gamma}\rightarrow K_{1}^{S,\Gamma}(A)\rightarrow 0,$

where $(GLF_{\ast}(A))^{\Gamma}$ is the Connes simplicial group of the $\Gamma$-simplicial group $GLF_{\ast}(A)$ induced by the  $\Gamma$-ring $A$.
\end{defn}

To define Swan's relative $\Gamma$-algebraic K-functors consider the augmented simplicial ring $F^{+}_{\ast}(A)=F_{\ast}(A)\overset{\tau}\rightarrow A$ for the $\Gamma$-ring $A$.

Let $f:A\rightarrow B$ be a surjective homomorphism of $\Gamma$-rings with unit. It induces a surjective morphism of augmented $\Gamma$-simplicial groups $GL(F^{+}_{\ast}(f):GL(F^{+}_{\ast}(A))\rightarrow GL(F^{+}_{\ast}(B))$, and therefore the morphism $(GL(F^{+}_{\ast}(f))_{\Gamma}:(GL(F^{+}_{\ast}(A)))^{\Gamma}\rightarrow (GL(F^{+}_{\ast}(B)))^{\Gamma}$.

 Define $K^{S,\Gamma}_{n}(A,I)=\pi_{n}(Ker((GL(F^{+}_{\ast}(f))_{\Gamma})), n> 2$, where $I=Ker f$. The groups $K_{1}^{S,\Gamma}(A,I)$ and $K_{2}^{S,\Gamma}(A,I)$ are defined by the exact sequence: $0\rightarrow K_{2}^{S,\Gamma}(A,I)\rightarrow  \pi_{0}Ker(GLF_{\ast}(f))^{\Gamma}\rightarrow Ker (GL(f))_{\Gamma}\rightarrow K_{1}^{S,\Gamma}(A,I)\rightarrow 0$. When the group $\Gamma$ is trivially acting on $A$, we recover the Swan's algebraic K-functors [17].

We also propose the definition of Quillen's $\Gamma$-algebraic K-functors by the following way.

Consider the free resolution $F_{\ast}(A)\overset{\tau}\rightarrow A$ of the unital $\Gamma$-ring $A$. The map $\tau$ induces a map of simplicial rings $\varphi_{\ast}: F_{\ast}\rightarrow (A)_{\ast}$, where $(A)_{\ast}$ is the constant simplicial ring, $A_{n}=A$ for all $n\geq 0$, and $\varphi_{n}= \tau \tilde{\lambda}^{n-1}_{0}, \tilde{\lambda}^{n-1}_{0}= \lambda^{1}_{0}\cdot\cdot\cdot\lambda^{n-1}_{0}$  for $n\geq 1$, and $\varphi_{0}=\tau$. Denote $I_{\ast}$ the kernel of $\varphi_{\ast}$ which is an acyclic ideal of $F_{\ast}$.

Let $A$ be a unital $\Gamma$-ring. The action of $\Gamma$ induces the action on the simplicial ring $I_{\ast}$, then on the simplicial group $GL(I_{\ast})$ and therefore the action on the space $BGL(I_{\ast})$ which is a homeomorphim, and if $f: [0,1]\rightarrow BGL(I_{\ast})$ is a continuous map, $(^{\gamma}f)(t)=^{\gamma}(f(t)), t\in  [0,1]$. This space could be used for the definition of Quillen's K-groups, since we have $K^{Q}_{n}(A)\cong \pi_{n-1}(BGL(I_{\ast})), n> 1$. The action of $\Gamma$  on $BGL(I_{\ast})$ induces its action on the homotopy groups of $BGL(I_{\ast})$,  therefore on $K^{Q}_{n}(A)$, and one has the isomorphism $(K^{Q}_{n}(A))_{\Gamma}\cong (\pi_{n-1}(BGL(I_{\ast})))_{\Gamma}, n> 1$.

Finally, we arrive to the definition of Quillen's algebraic $\Gamma$-functors $K^{Q,\Gamma}_{n}(A), n> 1$, as equal to the homotopy group $\pi_{n-1}(BGL(I_{\ast}) / \Gamma)$ of the quotient space $BGL(I_{\ast}) / \Gamma$ by the action of the discrete group $\Gamma$.

\begin{thm}
1. For a homomorphism $f:A\rightarrow B$ of $\Gamma$-rings with unit there is a long exact sequence

$\cdot\cdot\cdot\rightarrow K^{S,\Gamma}_{n+1}(B)\rightarrow K^{S,\Gamma}_{n}(A,I)\rightarrow K^{S,\Gamma}_{n}(A)\rightarrow K^{S,\Gamma}_{n}(B)\rightarrow \cdot\cdot\cdot \rightarrow  K^{S,\Gamma}_{2}(B)\rightarrow K^{S,\Gamma}_{1}(A,I)\rightarrow K^{S,\Gamma}_{1}(A)\rightarrow K^{S,\Gamma}_{1}(B)$,

 where $I=Ker f$, and one has a natural homomorphism $K^{S,\Gamma}_{n}(I)\rightarrow  K^{S,\Gamma}_{n}(A,I), n\geq 1$;

2. If A is a free ring on which a group $\Gamma$ is acting, then $K^{S,\Gamma}_{n}(A)=o$ for $n\geq 1$;

3. The relation between Swan's and Quillen's $\Gamma$-algebraic K-functors is expressed by the following exact sequence

$\pi_{n-2}(\Gamma GLF^{+}_{\ast}(A))\rightarrow (K^{Q}_{n}(A))_{\Gamma}\rightarrow K^{S,\Gamma}_{n}(A)\rightarrow \pi_{n-3}(\Gamma GLF^{+}_{\ast}(A))$,

$n> 1$, $F^{+}_{\ast}(A):F_{\ast}(A)\rightarrow A$ is the free resolution of the $\Gamma$-ring A;

4. There is the central exact sequence  $0\rightarrow K^{\Gamma}_{2}(A)\rightarrow K^{S,\Gamma}_{2}(A)\rightarrow Ker \nu_{\Gamma}\rightarrow e$, where $\nu_{\Gamma}$ is induced by the inclusion $\nu:E(A)\rightarrow GL(A)$.
\end{thm}

$Proof$. 1. Induced by the short exact sequence of augmented $\Gamma$-simplicial groups $e\rightarrow Ker((GL(F^{+}_{\ast}(f))_{\Gamma})\rightarrow (GL(F^{+}_{\ast}(A)))_{\Gamma}\rightarrow (GL(F^{+}_{\ast}(B)))_{\Gamma}\rightarrow e$; the homomorphism $K^{S,\Gamma}_{n}(I)\rightarrow  K^{S,\Gamma}_{n}(A,I)$ is induced by the composite of morphisms $(GL(F^{+}_{\ast}(I)))_{\Gamma}\rightarrow (GL(F^{+}_{\ast}(f))_{\Gamma}\rightarrow (GL(F^{+}_{\ast}(A)))_{\Gamma}$ for $n\geq 1$.

2. First it will be shown when A is $\Gamma$-free ring. In that case the augmented $\Gamma$-simplicial group $GLF^{+}_{\ast}(A)$ is left contractible and therefore all its homotopy groups are trivial. Let $F(X)$ be a free group with basis the subset $X$ on which the group $\Gamma$ is acting, then $F(F(X))$ is a $\Gamma$-free group, and consider the $\Gamma$-homomorphisms $F(X)\overset{\sigma}\rightarrow F(F(X))\overset{\tau}\rightarrow F(X)$ induced by the injection $X\rightarrow F(X)$ and the identity map $F(X)\rightarrow F(X)$ respectively. Therefore $K^{S,\Gamma}_{n}(F(X))$ is isomorphic to a subgroup of $K^{S,\Gamma}_{n}(F(F(X))$ which is trivial for $n\geq 1$.

3. By using the sequence 5.4 for the augmented $\Gamma$-simplicial group $GL(F^{+}_{\ast}(A))$ we obtain the following exact sequence $\pi_{n-2}(\Gamma GL(F^{+}_{\ast}(A))\rightarrow (\pi_{n-2}(GL(F^{+}_{\ast}(A)))_{\Gamma}\rightarrow \pi_{n-2}((GL(F^{+}_{\ast}(A)))_{\Gamma})\rightarrow \pi_{n-3}(\Gamma GL(F^{+}_{\ast}(A)))$, $n> 1$, implying the exact sequence
 $\pi_{n-2}(\Gamma GL(F^{+}_{\ast}(A))\rightarrow (K^{S}_{n}(A))_{\Gamma}\rightarrow K^{S,\Gamma}_{n}(A)\rightarrow \pi_{n-3}(\Gamma GL(F^{+}_{\ast}(A)))$, $n> 1$. One has the isomorphisms $(K^{Q}_{n})_{\Gamma}\cong (\pi_{n-1}BGL(I_{\ast}))_{\Gamma}\cong  (\pi_{n-2}GL(I_{\ast}))_{\Gamma}\cong (\pi_{n-2}GLF^{+}_{\ast}(A)_{\Gamma}\cong (K^{S}_{n}(A))_{\Gamma}, n> 1$, and we get the required exact sequence.

 4. It is easily checked $\pi_{0}((GL(F^{+}_{\ast}(A)))_{\Gamma})$ is isomorphic to $(St(A))_{\Gamma}$. Therefore $K^{S}_{2}(A)$ is isomorphic to the kernel of the homomorhism $(St(A))_{\Gamma}\rightarrow (GL(A))_{\Gamma}$, implying the required exact sequence which is central, since $K_{2}(A)$ is the center of $St(A)$. This completes the proof of the theorem.

\begin{rem}
Let A be a $\Gamma$-ring without unit and $A+\mathds{Z}$ be the unitization of A. The action of $\Gamma$ on A induces its action on $A+\mathds{Z}$ by $^{\gamma}(a,z)=(^{\gamma}a,z), a\in A, \gamma \in \Gamma$. Then there is the short exact sequence $0\rightarrow (K^{Q}_{n}(A))_{\Gamma}\rightarrow (K^{Q}_{n}(A+\mathds{Z}))_{\Gamma}\rightarrow (K^{Q}_{n}(\mathds{Z}))_{\Gamma}\rightarrow 0$ induced by the canonical $\Gamma$-homomorphisms $i:A\rightarrow A+\mathds{Z}$ and $p:A+\mathds{Z}\rightarrow \mathds{Z}$.
\end{rem}

For a commutative $\Gamma$-ring A with unit the product operation
$$
K^{\Gamma}_{1}(A) \otimes K^{\Gamma}_{1}(A)\rightarrow  K^{\Gamma}_{2}(A)
$$
is defined. For this aim we will follow Milnor's construction of the map $GL(A) \times GL(A)\rightarrow K_{2}(A)$ by using the symbol $[x,y]$, $x,y\in GL(A)$ given in [13]. Consider the canonical homomorphism $\tau: St(A)\rightarrow E(A)$ and let $x,y \in E(A)$ such that $xy = yx$. There is a correctly defined element
of $K_{2}(A)$ denoted $x\ast y$ by taking $v,z \in St(A)$ such that $\tau (v) = x, \tau (z) = y$. Then $x\ast y = vzv^{-1}z^{-1}$ and $\tau(vzv^{-1}z^{-1}) = e$.

Now the symbol $[x,y]$ is defined as follows (see [13]).
For $x\in GL_{m}(A)$ and $y\in GL_{n}(A)$ denote $x\otimes y$ the corresponding matrix to the induced automorphism of $A^{m}\otimes A^{n}$. Let I denote the $m\times m$ identity matrix and $I'$ be the $n\times n$ identity matrix. The matrices $diag(x\otimes I';x^{-1}\otimes I';I\otimes I')$ and $diag(I\otimes y,I\otimes I';I\otimes y^{-1})$ of $E_{3mn}(A))$ commute with each other, and the symbol $[x,y]$ is equal to
$$
diag(x\otimes I';x^{-1}\otimes I';I\otimes I')\ast diag(I\otimes y,I\otimes I';I\otimes y^{-1}) \in K_{2}(A).
$$
 The map $GL(A) \times GL(A)\rightarrow K_{2}(A)$ induced by the symbol $[x,y]$ is skew-symmetric, bimultiplicative, and gives rise to a homomorphism $K_{1}(A)\times K_{1}(A)\rightarrow K_{2}(A)$
[13]. Moreover, the following properties hold:

If u and $1-u$ are units, then $[u,1-u] = 1$ and $[u,-u] = 1$.

We are able now to show that the defined map $GL(A)\times GL(A)\rightarrow K_{2}(A)$ is a $\Gamma$-map if the ring A is a $\Gamma$-ring. If $\gamma \in \Gamma$, one has $^{\gamma}(x,y) = (^{\gamma}x,^{\gamma}y)$ and $^{\gamma}(x\ast y) = ^{\gamma}(vzv^{-1}z^{-1}) = ^{\gamma}{v}^{\gamma}{z}(^{\gamma}{v})^{-1}(^{\gamma}{z})^{-1} = (^{\gamma}x\ast ^{\gamma}y)$. Therefore,

$diag(^{\gamma}x\otimes I';(^{\gamma}x)^{-1}\otimes I';I\otimes I')\ast diag(I\otimes ^{\gamma}y,I\otimes I';I\otimes (^{\gamma}y)^{-1}) =  ^{\gamma}(diag(x\otimes I';x^{-1}\otimes I';I\otimes I'))\ast ^{\gamma}(diag(I\otimes y,I\otimes I';I\otimes y^{-1})) = ^{\gamma}(diag(x\otimes I';x^{-1}\otimes I';I\otimes I')\ast diag(I\otimes y,I\otimes I';I\otimes y^{-1}))$ showing that the homomorphism $K_{1}(A)\otimes K_{1}(A)\rightarrow K_{2}(A)$ is a $\Gamma$-homomorphism. This implies a homomorphism $(K_{1}(A)\otimes K_{1}(A))_{\Gamma}\rightarrow (K_{2}(A))_{\Gamma}$. Taking into account the isomorphism $K^{\Gamma}_{1}(A)\otimes K^{\Gamma}_{1}(A) \cong (K_{1}(A)\otimes K_{1}(A))_{\Gamma}$ and the canonical surjection $(K_{2}(A))_{\Gamma}\rightarrow K^{\Gamma}_{2}(A)$, we obtain the required homomorphism $K^{\Gamma}_{1}(A) \otimes K^{\Gamma}_{1}(A)\rightarrow  K^{\Gamma}_{2}(A)$.

The symbol $[x,y]$ is closely related to the Milnor's algebraic K-theory which we will define now.

Let A be a ring with unit. The n-th Milnor algebraic K-group $K^{M}_{n}(A)$ is defined as
$$
K^{M}_{n}(A) = (A^{\star})^{\otimes n}  / \{a_{1}\otimes a_{2} ... \otimes a_{n}: a_{i} + a_{i+1} = 1\},
$$
where $a_{i}\in A^{\star}$, $n\geq 2$ and $\{a_{1}\otimes a_{2} ... \otimes a_{n}\}$ is the Abelian subgroup generated by the elements $a_{1}\otimes a_{2} ... \otimes a_{n}$ satisfying the relation $a_{i} + a_{i+1} = 1$ called the Steinberg relation [14].

If a group $\Gamma$ is acting on the ring A, it induces a natural action on $K^{M}_{n}(A)$ given by its action on $(A^{\star})^{\otimes n}$ componentwise.  Therefore, if A is a $\Gamma$-ring,
 the nth Milnor's $\Gamma$-algebraic K-functor $K^{M,\Gamma}_{n}(A)$ is defined as
$$
((A^{\star})^{\otimes n} / \{a_{1}\otimes a_{2} ... \otimes a_{n}: a_{i} + a_{i+1} = 1\})_{\Gamma},
$$
where $a_{i}\in A^{\star}$ and $n\geq 2$.

Matsumoto proved that the homomorphism $A^{\star}\times A^{\star} / a\otimes (1-a) \rightarrow K_{2}(A)$, $a\neq 0,1$, induced by the symbol $[u,v ]$, is an isomorphism, if A is a field and as we have shown, it is a $\Gamma$-map.
 As a consequence, we obtain
 \begin{thm}
Matsumoto's theorem for $\Gamma$-fields.

 If F is a $\Gamma$-field, one has the isomorphism

$K^{M,\Gamma}_{2}(F) \cong (K_{2}(F))_{\Gamma}$ and the exact sequence

$0\rightarrow (K_{2}(F)\cap \Gamma St(F)) / \Gamma K_{2}(F)\rightarrow K^{M,\Gamma}_{2}(F)\rightarrow K^{\Gamma}_{2}(F)\rightarrow 0$.

\end{thm}

It is well-known (14[,18]) that the Milnor algebraic K-theory is closely related to the Witt ring of quadratic forms and to
the Bloch higher Chow groups. Now we will attempt to show these relations for the case of $\Gamma$-fields by considering the Milnor's conjectures.

Let $F$ be a field of characteristic different from 2. The Witt ring $W(F)$ is consisting of equivalence classes of non-degenerate quadratic forms. It has a filtration
$$
I(F)\supset I^{2}(F)\supset ... \supset I^{n}(F)\supset ...
$$
given by the powers $I^{n}(F)$ of the fundamental ideal $I(F)$ of even-dimensional quadratic forms. The quotients $I^{n}(F) / I^{n+1}(F)$ play an important role in the study of the Witt ring $W(F)$.

Milnor conjectures the isomorphism $\kappa^{M}_{n}(F)\cong I^{n} / I^{n+1}$, where $\kappa^{M}_{n}(F)$ is  the n-th Milnor K-group modulo 2 of the field $F$, proved in [18].

For the Milnor conjecture about Chow groups, first we will give the definition of the Chow group $CH^{k}(F,n)$ of the field $F$.

Let $V$ be a vector space over the field $F$. The projective space $P(V)$ is the set of equivalence classes of $V\backslash 0$ under the relation $\sim$ defined by $x\sim y$ if there is non zero element $\lambda\in F$ such that $x=\lambda y$. If $V=F^{n+1}$, the relevant projective space is denoted $P^{n}_{F}$,
$$
  P^{n}_{F}= \{a_{1}, ... ,a_{n+1}\}/\sim, a_{i}\in F,
  $$
not all $a_{i}=0$. and $(a_{1}, ... , a_{n+1})\sim (a'_{1}, ... , a'_{n+1})$ if there is a non zero element $\lambda\in F$ such that $a'_{i}=\lambda a_{i}, i=1,...,n+1.$

Let $Z_{k}(F,n)$ be the free abelian group on the set of $k$- dimensional subvarieties of $P^{n}_{F}$ called the group of $k$- dimensional algebraic cycles on $P^{n}_{F}$. Denote $B_{k}(F,n)$ the subgroup of $k$-cycles rationally equivalent to zero. Then the Chow group $CH^{k}(F,n)$ is defined as $Z_{k}(F,n)/B_{k}(F,n)$. It is clear the Chow group $CH^{k}(F,n)$ is trivial for $k\succ n$.

Milnor conjectures the isomorphism $K^{M}_{n}(F)\cong CH^{n}(F,n)$, proved in [15,19,20].

Let $F$ be a $\Gamma$-field. Denote $\kappa^{M,\Gamma}_{n}(F)$ the n-th Milnor K-group modulo 2 and $CH^{n,\Gamma}(F,n)$ the n-th Chow group of the $\Gamma$-field $F$ which will be defined a bit later.

The Milnor conjectures for the $\Gamma$-field $F$ take the form:

\begin{thm}
If $F$ is a $\Gamma$-field, there are isomorphisms

1. $\kappa^{M,\Gamma}_{n}(F)\cong (I^{n} / I^{n+1})_{\Gamma},$

2.  $K^{M,\Gamma}_{n}(F)\cong CH^{n,\Gamma}(F,n).$
\end{thm}
$Proof.$ For the first isomorphism  we have to define the group $\kappa^{M,\Gamma}_{n}(F)$ and the action of $\Gamma$ on $I^{n} / I^{n+1}.$ The action of $\Gamma$ on $K^{M}_{n}(F)$ induces its action on $\kappa^{M}_{n}(F) = K^{M}_{n}(F) / 2K^{M}_{n}(F)$, and define $\kappa^{M,\Gamma}_{n}(F)$ as equal to $(\kappa^{M}_{n}(F))_{\Gamma}.$

The ideal $I^{n}(F)$ is additively generated by the $n$-fold Pfister forms $<< a_{1}, ... ,a_{n} >> = \bigotimes^{n}_{i=1} <1, -a_{i}>$ which is a quadratic form of dimension $2^{n}$, $a_{i}\in F^{\star}$. For instance, the 1-fold and 2-fold Pfister forms are respectively  $<<a>>\cong <1, -a> = x_{1}^{2}-ax_{2}^{2}$ and $<<a,b>> \cong <1,-a,-b-,ab> = x_{1}^{2}-ax_{2}^{2}-bx_{3}^{2}+abx_{4}^{2}$.

If $F$ is a $\Gamma$-ring, the action of the group $\Gamma$ on the field $F$ induces a natural action of $\Gamma$ on the generators of the power $I^{n}(F)$ given by $^{\gamma}(<< a_{1}, ... ,a_{n} >>) = << ^{\gamma}a_{1}, ... ,^{\gamma}a_{n} >>$, $\gamma \in \Gamma$. The above given filtration of the powers of the fundamental ideal $I(F)$ becomes a filtration of $\Gamma$-groups. It is easily checked that this action is compatible with the equivalence of quadratic forms.

For $\Gamma$-fields the quotient corresponding to the quotient $I^{n} / I^{n+1}$ would be its equivariant form $(I^{n} / I^{n+1})_{\Gamma}$ which is equal to $I^{n} / I^{n+1}\cdot \Gamma I^{n}$.

There is a homomorphism $\alpha^{M}_{n}$ from $\kappa^{M}_{n}(F)$ to the quotient $I^{n} / I^{n+1}$ defined by Milnor and given by $(a_{1}, ... ,a_{n})\rightarrow << a_{1}, ... ,a_{n} >>$.

 We need to show that the homomorphism $\alpha^{M}_{n}$ is a $\Gamma$-map. In effect, it suffices to show this fact on the set of generators $\{(a_{1}, ... ,a_{n}), a_{i}\in F^{\star}\}$. One has $^{\gamma}(a_{1}, ... ,a_{n}) = (^{\gamma}a_{1}, ... ,^{\gamma}a_{n}) = << ^{\gamma}a_{1}, ... ,^{\gamma}a_{n} >> = ^{\gamma}(<< a_{1}, ... ,a_{n} >>)$. This implies the isomorphism

$$
(K^{M}_{n}(F) / 2K^{M}_{n}(F))_{\Gamma}\cong (I^{n} / I^{n+1})_{\Gamma}
$$
and the proof of the first Milnor conjecture.

Regarding the second Milnor conjecture, we will define the action of the group $\Gamma$ on the Chow group $CH^{k}(F,n)$ induced by its action on the field $F$. It is clear that it induces an action on the projective space $P^{n}_{F}$ defined by $^{\gamma}(a_{1}, ... ,a_{n}) = (^{\gamma}a_{1}, ... ,^{\gamma}a_{n})$. This action induces an action on the set of $k$-dimensional subvarieties and therefore, on the groups $Z_{k}(F,n)$ and $B_{k}(F,n)$. Finally, the action of $\Gamma$ on the Chow group $CH^{k}(F,n)$ is induced by its action on $Z_{k}(F,n)$. The $\Gamma$-Chow group $CH^{k,\Gamma}(F,n)$ is defined as equal to $(CH^{k}(F,n))_{\Gamma}$.

There is a well-known homomorphism $\beta^{M}_{n}:K^{M}_{n}(F)\rightarrow CH^{n}(F,n)$ and Milnor raised the homomorphism $\beta^{M}_{n}$ is an isomorphism [14]. For its definition we need the product map of Chow groups. The product map $CH^{p}(F,q)\times CH^{r}(F,s)\rightarrow CH^{p+r}(F,q+s)$ is given by intersecting algebraic cycles, namely the product $|y| |z|$ in $CH^{p+r}(F,q+s)$ is the sum of the varieties of the intersection $Y\bigcap Z$ which all have codimension $q+s$, where $Y$ and $Z$ are subvarieties of $(F,q)$ and $(F,s)$ respectively. It is easily checked that the product map is a $\Gamma$-map with componentwise action of $\Gamma$ on the product.

The homomorphism $\beta^{M}_{n}$ is induced by the homomorphism $\beta^{M}_{1}:F^{\ast}\rightarrow CH^{1}(F,1)$ of abelian groups, sending the unit $e$ of $F^{\ast}$ to $0\in CH^{1}(F,1)$, and the element $a, a\neq e,$ to the class of the element $(a,1)$ of the projective space $P^{1}_{F}$. Since $\beta^{M}_{1}$ and the product map are $\Gamma$-maps, it follows that the induced homomorphism $\beta^{M}_{n}$ is a $\Gamma$-homomorphism. This implies the isomorphism
$$
(K^{M}_{n}(F))_{\Gamma} \cong (CH^{n}(F,n))_{\Gamma}
$$
and completes the proof of Milnor's conjectures for $\Gamma$-fields.

\begin{cor}
For $\Gamma$-fields the relationship between Chow groups and Milnor K-functor $K_{2}$ takes the following form:
$$
0\rightarrow (K_{2}(F)\cap \Gamma St(F)) / \Gamma K_{2}(F)\rightarrow CH^{n,\Gamma}(F,n)\rightarrow K^{\Gamma}_{2}(F)\rightarrow 0.
$$
\end{cor}
Follows immediately from Theorems 5.18 and 5.19.
\;

\section{Acknowledgment}
I would like to thank the referee for carefully reading the manuscript
and for giving such constructive remarks which substantially helped
improving the quality of the presentation of the paper.
\;




\begin{bibdiv}
\begin{biblist}




\bib{Ad}{article}{
author={Adamson I.T},
title={Cohomology theory for non-normal subgroups and non-normal fields},
journal={Glasgow Math. Ass.},
volume={1954(2)},
date={1954},
pages={66-76},
}	

\bib{ad}{book}{
author={Bass H.},
title={Algebraic K-theory},
Publisher={W.A.Benjamin, Inc.},
date={1968},
}



\bib{GrMePa}{article}{
author={Grant M.},
author={Meir E.},
author={Patchkoria I.},
title={Equivariant dimensions of groups with operators},
journal= {Groups, Geometry, and Dynamics},
volume={16}
pages={1049-1075}
date={2022},

}

\bib{GrMePa}{article}{
author={Grant M.},
author={Li K.},
author={Meir E.},
author={Patchkoria I.},
title={Comparison of equivariant cohomological dimensions},
journal= {arXiv},
volume={2302.08574v1}
pages={1049-1075}
date={2023},

}



\bib{Gu}{article}{
author={Guin D.},
title={Cohomologie et homologie non abéliennes des groupes},
journal={J. Pure Appl. Algebra},
volume={50},
date={1088},
pages={109-137},
}

\bib{Ho}{article}{
author={Hochschild G.},
title={Relative homological algebra},
journal={Trans. Amer. Math. Soc.},
volume={82},
date={1956},
pages={246-269},
}





\bib{In}{book}{
author={Inassaridze H.},
title={Non-Abelian Homological Algebra and its Applications},
Publisher={Kluwer Academic, Dordrecht},
date={1997},

}

\bib{In}{article}{
author={Inassaridze H.},
title={Equivariant homology and cohomology of groups},
journal={Topology and its Applications },
volume={153},
date={2005},
pages={66-89},
}

\bib{InKh}{article}{
author={Inassaridze H.},
author={Khmaladze E.}
title={Hopf formulas for equivariant integral homology of groups},
journal={Proc. Amer. Math. Soc.},
volume={138(9)},
date={2010},
pages={3037-3046},
}




\bib{In}{article}{
author={Inassaridze H.},
title={(Co)homology of $\Gamma$-groups and $\Gamma$-homological algebra},
journal={European Journal of Mathematics},
volume={8(Suppl 2)},
date={2022},
pages={5720-5763},
}

\bib{Lo}{book}{
author={Loday J-L.},
title={Cyclic Homology},
Publisher={Springer, Grundlehren der mathematischen Wissenschaften 301},
date={1998},

}


\bib{Ma}{article}{
author={Matsumoto H.},
title={Sur les sous groupes arithmétiques des groupes semi-simples déployés},
journal={Ann. Sci. Norm. Sup.},
volume={2(4)},
date={1969},
pages={1-62},
}





\bib{Mi}{book}{
author={Milnor J.},
title={Introduction to Algebraic K-theory},
Publisher={Annals of Mathematic Studies 72, Princeton University Press, NJ},
date={1971},

}


\bib{Mi}{article}{
author={Milnor J.},
title={Algebraic K-theory and Quadratic Forms},
Publisher={},
volume={9}
date={1970},
pages={318-344},
}

\bib{OrViVo}{article}{
author={Orlov D.},
author={Vishik A.},
author={Voevodsky D.},
title= {An exact sequence for $K^{M}_{\ast} / 2$ with applications to quadratic forms},
Publisher={Annals of Mathematics},
volume={165},
date={2007},
pages={1-13},

}


\bib{Qu}{article}{
author={Quillen D.},
title={Higher algebraic K-theory},
journal={Lecture Notes in Math. Springer},
volume={341},
date={1973},
pages={77-139},

}


\bib{Sw}{book}{
author={Swan R.G.},
title={Algebraic K-theory},
Publisher={Lecture Notes in Mathematics, 76},
date={1968},

}


\bib{To}{article}{
author={Totaro B.},
title={Milnor K-Theory is the Simplest Part of Algebraic K-Theory},
journal={K-Theory},
volume={6}
date={1992},
pages={177-189},
}


\bib{Vo}{article}{
author={Voevodsky V.},
title={Motivic cohomology with Z/2-coefficients},
journal={Publ. Math. Inst. Hautes Etudes Sci.},
volume={98}
date={2003},
pages={59-104},
}


\bib{Vo}{article}{
author={Voevodsky V.},
title={Reduced power operations in motivic cohomology},
journal={Publ. Math. Inst. Hautes Etudes Sci.},
volume={98}
date={2003},
pages={1-57},
}









\end{biblist}
\end{bibdiv}
\end{document}